\documentclass[a4paper, 10pt]{amsart}

\usepackage{amsmath,amssymb}
\usepackage{xcolor}
\usepackage{mathabx}
\usepackage[margin = 2.5cm]{geometry}
\usepackage[english]{babel}
\usepackage{amsthm}
\usepackage{tikz-cd}
\usepackage[]{mathrsfs}
\usepackage{mathtools}
\usepackage{todonotes}
\usepackage{calc}
\usepackage{graphicx}
\usepackage{subfigure}
\usepackage{faktor}
\usepackage{amsfonts}
\usepackage{hyperref}
\hypersetup{colorlinks=true, linkcolor=red}
\usepackage{verbatim}
\usepackage{pinlabel}
\usepackage[T1]{fontenc}
\usepackage{graphicx}
\usepackage{graphbox}
\usepackage{caption}
\usepackage{subcaption}
\usepackage{enumitem}

\usepackage{setspace}

\usepackage{url}

\usepackage{soul}

\newtheorem{theorem}{Theorem}[section]
\newtheorem{corollary}[theorem]{Corollary}
\newtheorem{proposition}[theorem]{Proposition}
\newtheorem{definition}[theorem]{Definition}
\newtheorem{lemma}[theorem]{Lemma}
\newtheorem{claim}[theorem]{Claim}

\newtheorem*{theorem*}{Theorem}
\newtheorem*{proposition*}{Proposition}
\newtheorem*{definition*}{Definition}
\newtheorem*{lemma*}{Lemma}
\newtheorem*{claim*}{Claim}
\newtheorem*{corollary*}{Corollary}
\newtheorem*{convention*}{Convention}

\newtheorem{thmintro}{Theorem}

\theoremstyle{definition}

\newtheorem*{question*}{Question}
\newtheorem*{notation}{Notation}
\newtheorem{construction}{Construction}

\theoremstyle{remark}
\newtheorem{rem}[theorem]{Remark}
\newtheorem*{rem*}{Remark}
\newtheorem*{acknowledgement}{Acknowledgement}
\newtheorem{case}{Case}

\newcommand{\Z}{\mathbb{Z}}

\newcommand{\R}{\mathbb{R}}

\newcommand{\T}{\mathbb{T}}

\newcommand{\tor}{\mathrm{Tor}}

\newcommand{\cP}{\mathcal{P}}

\newcommand{\cF}{\mathcal{F}}
\newcommand{\PP}{\mathcal{P}}

\newcommand{\acts}{\curvearrowright}

\newcommand{\flow}{\varphi}

\newcommand{\orbitspace}{(P_\varphi, \widetilde{\cF}^s, \widetilde{\cF}^u)}

\newcommand{\overbar}[1]{\mkern 1.5mu\overline{\mkern-1.5mu#1\mkern-1.5mu}\mkern 1.5mu}

\newcounter{notes}

\title{Transitive Anosov flows on non-compact manifolds}

\author[T.~Barthelm\'e]{Thomas Barthelm\'e}
\address{Queen's University, Kingston, Ontario}
\email{thomas.barthelme@queensu.ca}
\urladdr{sites.google.com/site/thomasbarthelme}

\author[L.~Lu]{Lingfeng Lu}
\address{Queen's University, Kingston, Ontario}
\email{21ll24@queensu.ca}

\subjclass[2020]{Primary 37D20, 37D05;
  Secondary 37C27}

\begin{document}

\begin{abstract}
    In this article we study topological transitivity of Anosov flows on non-compact 3-manifolds. We provide both homological and homotopical conditions under which lifts to infinite covers of transitive Anosov flows stay transitive. In particular, we show that most transitive Anosov flows admit a transitive lift to an infinite cover, and we build a family of examples of transitive Anosov flows on non-compact manifolds that satisfy a homotopical characterization of suspension flows. 
\end{abstract}

\maketitle

\section{Introduction}

The topological study of Anosov flows, or Anosov diffeomorphisms, has been so far almost exclusively limited to the case where the ambient manifold is compact. Some exceptions to that rule can be found in the works on Anosov diffeomorphisms on the plane or open surfaces, e.g., \cite{Groisman_Nitecki_Anosov_diffeo_plane,Matsumoto21,Mendes77,ovadia2024anosovdiffeomorphismsopensurfaces}, as well as recently, in the introduction and study of \emph{Anosov-like} actions \cite{BFM2022orbit,BBM24}, where compactness is not explicitly needed.

The dynamical side of the study of Anosov flows has seen many more works where the ambient manifold is not compact, for instance in the very natural context of studying the properties of geodesic flows on non-compact, negatively curved manifolds. Another instance where the compactness assumption is explicitly dropped is, for instance, in the work of Sharp \cite{Sharp_1993} and Dougall--Sharp \cite{DS21}, where they prove counting results on (infinite) abelian covers of Anosov flows.

One of the issues appearing when considering Anosov dynamics in the non-compact setting is that the non-wandering set, i.e., the part of the manifold where the interesting dynamics appear, can be small or even trivial. (A classical example is that there exists a complete metric making the translation on $\R^2$ Anosov \cite{White73}.) So for the dynamics to be interesting, one would like to be able to ensure that the non-wandering set is large enough.
In this article, we consider regular infinite covers (see Definition \ref{def: covers}) of compact manifolds, as they are a very large and natural source of examples of Anosov flows in the non-compact setting and address the following general question:
\begin{question*} 
    Let $\flow$ be a transitive Anosov flow on a compact $3$-manifold $M$. Let $\overbar{M}$ be any infinite cover of $M$ and $\overbar{\flow}$ the lifted flow. How common is it for $\overbar{\flow}$ to be transitive?
\end{question*}

This question of existence and/or abundance of transitive Anosov flows in non-compact manifolds, or infinite covers of compact manifolds has appeared more or less explicitly in several different works. First, in the context of geodesic flows for negatively curved metrics, the question of transitivity (both in the compact and non-compact case), was an historically important one that was resolved in the 1930s (see \cite{Hedlund39}) in low dimension and, in greater generality and any dimensions, in the early 1970s \cite{Eberlein72}.
Much more recently, Dougall and Sharp consider in \cite{DS21} certain Abelian covers $Y$ of an Anosov flow $\flow$ on $M$ and \emph{assume} the transitivity of the lifted flow $\flow_Y$, making it natural to wonder whether or not this assumption is necessary\footnote{Contrarily to this article, the dimension of the manifolds they consider in \cite{DS21} is arbitrary.}. In \cite[Question 3.1]{RodriguezHertz2008}, they ask about the existence of transitive Anosov diffeomorphisms on non-compact manifolds. And in \cite{BFM2022orbit,BBM24}, while the axiomatic definition of Anosov-like actions does not require (co-)compactness, compactness of the manifold is used to prove that (pseudo-)Anosov flows gives rise to Anosov-like action, making it natural to ask whether or not compactness is necessary.

In this article, we give answers to (some versions of) each of these questions. Before stating these more precisely, we start with the following result, showing that the existence of transitive Anosov flows on infinite covers is almost ubiquitous:
\begin{thmintro}\label{thmintro: ubiquitous existence}
    Let $M$ be a compact 3-manifold that admits a transitive Anosov flow. If $M$ is not a graph manifold\footnote{A \emph{graph manifold} is a $P^2$-irreducible (irreducible and contains no two-sided $\R \mathrm{P}^2$'s) $3$-manifold such that its JSJ decomposition consists only of Seifert fibered pieces, see e.g., \cite{Hat23}.}, then there exists an infinite cover of $M$ on which the lifted flow is transitive.
\end{thmintro}
Our result is in fact more general than this, as for many cases of Anosov flows on graph manifolds, we can still exhibit infinite covers where the lifted flow is transitive. See Theorems \ref{theorem: rfrs cover}, \ref{theorem: R-covered transitivity} and Corollary \ref{corollary: rfrs manifolds} for more complete statements.

Additionally from the general existence question, we provide both homological and homotopical conditions ensuring transitivity of flows in certain covers, as well as examples of transitive flows on non-compact manifolds satisfying extra properties.

\subsection{Transitivity on abelian covers: A homological condition}
In \cite{Sharp_1993}, Sharp studied \emph{homologically full} Anosov flows that can be defined as transitive Anosov flows $\flow$ on $M$ such that every element of $H_1(M;\Z)$ is represented by a periodic orbit of $\flow$. 
In \cite{DS21}, they extended this definition in the following way: Given $Y\to M$ a regular abelian cover of $M$ with associated covering group $G$, then $G$ can be identified with a quotient of $H_1(M;\Z)$, and they call a transitive Anosov flow on $M$ \emph{$Y$-full} if every class in $G$ is represented by a periodic orbit of $\flow$. It is shown in \cite[Lemma 7.4]{DS21} that if the lift $\flow_Y \colon Y \to Y$ is transitive then $\flow$ is $Y$-full.

Gogolev and Rodriguez Hertz proved, in \cite[Theorem 2.5]{gogolev2020abelianlivshitstheoremsgeometric}, the converse of that statement when $Y$ is the maximal abelian cover. We extend their argument to any abelian cover, thus obtaining the following homological condition for transitivity.

\begin{thmintro}\label{thmintro: Y-full}
Let $\flow$ be a transitive Anosov flow on a compact $3$-manifold $M$ and $Y$ be a regular abelian cover of $M$, then the lifted flow $\flow_Y$ on $Y$ is transitive if and only if $\flow$ is $Y$-full.
\end{thmintro}
More importantly, the argument of Gogolev--Rodriguez Hertz allows us to give an easily checkable condition for building a cover $Y$ for which $\flow$ is $Y$-full, see Theorem \ref{theorem: construct Y-full cover}. 
Theorem \ref{thmintro: ubiquitous existence} is a consequence of this, together with very deep results in $3$-manifolds topology due to Agol \cite{agol_virtual_haken_conjecture} and Przytycki--Wise \cite{wise_mixed_3_manifolds}, that allows one to always build a finite cover of $M$ such that the lifted flow is $Y$-full for some further infinite abelian cover $Y$, provided $M$ is not a graph-manifold.

\subsection{Transitivity on regular covers: the $\R$-covered case}

An Anosov flow $\flow$ on a $3$-manifold $M$ is called $\R$-covered if the leaf space of its stable, or unstable, foliation is homeomorphic to $\R$ (see Section \ref{section: orbit space} for details). This class of Anosov flows classically contains all suspension of Anosov diffeomorphisms on the torus, geodesic flows of negatively curved surfaces, as well as many more examples, for instance all Anosov flows preserving a contact structure.
For this class of Anosov flows, we have the striking result that their lifts to \emph{all} regular covers (not just abelian ones) are either transitive or all orbits are wandering:
\begin{thmintro}\label{thmintro: R-covered}
    Let $\varphi \colon M \to M$ be an $\R$-covered Anosov flow. Let $\overbar{M}$ be a regular cover of $M$ and $\overbar{\varphi}$ be the lift of $\varphi$ to $\overbar{M}$. Then $\overbar{\varphi}$ is transitive if and only if $\overbar{\varphi}$ admits a periodic orbit, and if and only if it has a dense set of periodic orbits.
    
    Moreover, if $\overbar{\varphi}$ is \emph{not} transitive, then we have exactly one of the three following cases:
    \begin{enumerate}
        \item $\overbar{M} = \widetilde{M}$, i.e., it is the universal cover;
         \item $\overbar{M}$ is a fiberwise cover of $T^1\widetilde{\Sigma}$, where $\widetilde{\Sigma} \cong \mathbb{H}^2$ is the universal cover of some hyperbolic orbifold $\Sigma$, and $\varphi$ is a lift of the geodesic flow on $T^1\Sigma$;
        \item $\overbar{M} = \mathbb{T}^2 \times \R$ or $\overbar{M} = \mathbb{S}^1 \times \R \times \R$, and $\varphi$ is the suspension flow of an Anosov diffeomorphism on a torus.
    \end{enumerate}
\end{thmintro}
In particular this theorem extends the classical result of transitivity on regular cover of geodesic flows on negatively curved surfaces \cite{Hedlund39} to the much wider class of $\R$-covered Anosov flows.
Our proof uses the recent work of the first author with C.~Bonatti and K.~Mann in \cite{BBM24} about non-transitive Anosov-like actions. Theorem \ref{thm: transitive cover no infinite chains} gives another consequence of that work for some non-$\R$-covered flows.

\subsection{Transitive Anosov flows and the homotopical characterization of suspension flows}

As mentioned above, in \cite[Question 3.1]{RodriguezHertz2008}, F.~Rodriguez Hertz, J.~Rodriguez Hertz, and R.~Ures asks about the existence of transitive Anosov diffeomorphisms on non-compact manifolds in general and on surfaces in particular. This question came naturally in their setting, as, if one could prove that such examples did not exists, one could improve their result about accessibility of certain partially hyperbolic diffeomorphisms,  see \cite[Theorem 1.3]{RodriguezHertz2008}. 

As far as we are aware, there has not been a lot of progress on this question, except for a result showing the non-existence of such examples under additional regularity assumptions; see \cite{ovadia2024anosovdiffeomorphismsopensurfaces}.

While we do not answer that question for Anosov diffeomorphisms here, we answer one version for flows, showing that there exists (a lot of) Anosov flows on non-compact manifolds that satisfy what is known to be a homotopical characterization of suspensions in the compact case. More precisely, a result due to the work of Barbot and Fenley (see e.g., \cite[Theorem 2.15]{BBGRH_anomalous} for a proof) provides the following characterization of suspensions. 
Before stating it, recall that two flows $\varphi, \psi \colon M \to M$ are said to be \emph{orbit equivalent} if there is a homeomorphism of $M$ that takes the orbits of $\varphi$ to the orbits of $\psi$.

\begin{theorem}[Barbot, Fenley]\label{thm: homotopical charachterization suspension}
    Let $\flow$ be an Anosov flow on a compact $3$-manifold $M$. Then $\flow$ is orbit equivalent to the suspension of an Anosov diffeomorphism on the torus if and only if no two periodic orbits $\alpha$, $\beta$ of $\flow$ are freely homotopic (via a nontrivial homotopy if $\alpha=\beta$).
\end{theorem}

We show
\begin{thmintro}\label{thmintro: noncompact_suspension}
    There exists transitive Anosov flows $\overbar{\flow}$ on non-compact $3$-manifolds $\overbar{M}$ such that every periodic orbit of $\overbar{\flow}$ is alone in its free homotopy class.
\end{thmintro}
Note that if one can show that such examples admit a surface of section, then it would provide an example that answer the question of F.~Rodriguez Hertz, J.~Rodriguez Hertz, and R.~Ures.

Such examples also gives examples of actions that satisfy all the conditions of a topologically transitive Anosov-like action except for one. See Theorem \ref{theorem: rfrs cover} and Section \ref{section: non-compact suspension like}.

\begin{rem}
    Theorems \ref{thmintro: ubiquitous existence} and \ref{thmintro: Y-full} can easily (often with no or only cosmetic changes to the arguments) be extended to pseudo-Anosov flows, but we chose to stay within the realm of Anosov flows in this article.
\end{rem}

\subsection*{Outline}
In Section \ref{section: orbit space}, we recall some definitions and results on Anosov-like actions, with a focus on those induced by Anosov flows on compact 3-manifolds. 

In Section \ref{section: rfrs covers}, we prove Theorem  \ref{thmintro: Y-full}, as well as Theorem \ref{theorem: construct Y-full cover} which gives a simple homological condition for a flow to be $Y$-full. Finally, we prove Theorem \ref{thmintro: ubiquitous existence}.  

Section \ref{section: regular covers} is dedicated to the proof of Theorem \ref{thmintro: R-covered} and its generalization, Theorem \ref{thm: transitive cover no infinite chains}.

Finally, in Section \ref{section: non-compact suspension like}, we construct a family of examples to prove Theorem \ref{thmintro: noncompact_suspension}.

\begin{acknowledgement}
    The authors thank Katie Mann for many helpful discussions. Both authors were partially supported by the NSERC grants ALLRP 598447 - 24 and RGPIN-2024-04412.
\end{acknowledgement}


\section{Anosov-like actions and structure of the orbit space} \label{section: orbit space}

Recall that a (pseudo-)Anosov flow $\flow$ on a 3-manifold $M$ induces an action of $\pi_1(M)$ on the \emph{orbit space} $P_\flow$ of $\flow$. The orbit space is defined as the quotient of the universal cover $\widetilde{M}$ of $M$ under the relation ``being on the same orbit of the lifted flow $\widetilde\flow$''. A fundamental result of Barbot \cite{barbot_1995} and Fenley \cite{fenley1994_anosov_flow} is that the orbit space is a topological plane. The stable and unstable foliations $\cF^s$ and $\cF^u$ associated to $\varphi$ lift to foliations $\widetilde{\cF}^s$, $\widetilde{\cF}^u$ of $\widetilde \flow$, which then descend to a pair of transverse foliations of $P_\flow$, giving it a structure of a bifoliated plane $\orbitspace$. Moreover, the induced action of $\pi_1(M)$ on the orbit space is an example (see \cite[Proposition 2.2]{BBM24}) of the general class of \emph{Anosov-like} action introduced in \cite{BFM2022orbit,BBM24}.

\begin{definition}\label{def: anosov-like}
    A group $G$ acts \emph{Anosov-like} on a bifoliated plane $\PP = (P, \cF^+, \cF^-)$ (possibly with isolated prong singularities) if the action satisfies all of the following axioms:
    \begin{enumerate}[label = \textcolor{red}{({A}\arabic*)}]
        \item \label{anosov-like axiom hyperbolic} If a nontrivial element $g \in G$ fixes a leaf $l \in \cF^\pm$, then it has a unique fixed point $x \in l$ and its action is topologically contracting on one leaf through $x$ and topologically expanding on the other. 
        
        \item \label{anosov-like axiom dense fixed leaves} The set of $\cF^+$-leaves that are fixed by some nontrivial element of $G$ is dense in $\PP$, as is the set of $\cF^-$-leaves that are fixed by some nontrivial element of $G$.

        \item \label{anosov-like axiom singularity} Each singular point is fixed by some nontrivial element of $G$.

        \item \label{Anosov-like axiom branching} If $l \in \cF^\pm$ is a leaf that is nonseparated from another leaf $l^\prime$ in the corresponding leaf space, then $l$ is fixed by some nontrivial element of $G$.

\end{enumerate}
\end{definition}

We refer to \cite{BFM2022orbit,BBM24} and the monograph \cite{BM_book} for more background on Anosov-like actions. Note  that the definition taken in \cite{BFM2022orbit} only dealt with the transitive case,  with the non-transitive definition given in \cite{BBM24}. Moreover, both \cite{BFM2022orbit} and \cite{BBM24} ask for an additional axiom that is not needed here. See \cite{BM_book} for more details. 

Among the axioms in Definition \ref{def: anosov-like}, \ref{anosov-like axiom hyperbolic}--\ref{anosov-like axiom singularity} reflect the fundamental properties of $\pi_1(M) \acts \orbitspace$ such as hyperbolicity and density of the leaves of periodic orbits. However, the proof of \cite{fenley98_structure_of_branching} that the induced action of a pseudo-Anosov flow satisfy Axiom \ref{Anosov-like axiom branching} relies in an essential way on the compactness of the 3-manifold $M$.

When the Anosov-like action is one induced by a pseudo-Anosov flow, the transitivity of the flow can be read from the action: Smale's decomposition theorem implies that $\flow$ is transitive if and only if $\pi_1(M)\acts (P_\flow,\widetilde\cF^s,\widetilde\cF^u)$ is topologically transitive, or equivalently if and only if the set of points that are fixed by some elements of $\pi_1(M)$ is dense in $\orbitspace$ (see, e.g., \cite[Theorem 5.3.50]{fisher_hasselblatt_hyperbolic_flows} for the smooth Anosov case and  \cite{BBM24} for general topological pseudo-Anosov flows). More generally, work of \cite{BBM24} implies that one always has the following equivalences:
\begin{theorem}[\cite{BM_book}, Theorem 2.7.2]\label{thm_characterization_transitive}
Suppose a group $G$ acts Anosov-like on a bifoliated plane $\cP = (P, \cF^+, \cF^-)$. The following are equivalent:
\begin{enumerate}[label=(\roman*)]
\item \label{item_charac_top_transitive} The action of $G$ is topologically transitive;
\item \label{item_charac_dense_fixed_points} The set of points fixed by nontrivial elements of $G$ is dense in $\cP$;
\item \label{item_charac_dense_leaves} For any leaf $l$ of $\cF^\pm$, $G\cdot l$ is dense in $\cP$.
\end{enumerate}
\end{theorem}

We further recall some definitions and results regarding the configurations of leaves when a bifoliated plane admits an Anosov-like action. 

\begin{definition}
    If $\cF$ is a (possibly singular) foliation of a manifold, then the \emph{leaf space} of $\cF$, denoted $\Lambda(\cF)$, is the quotient space obtained by identifying each leaf of $\cF$ to a single point, equipped with the quotient topology.
\end{definition}

It is a classical result, originally due to the work of Reeb and Haefliger, that the leaf space of a (non-singular) foliation of a plane is a connected, simply-connected, non-compact, and not necessarily Hausdorff 1-manifold; see e.g., \cite{candel2000foliations}. In the context of Anosov flows, we have the following standard terminology.  \par

\begin{definition}
    Let $\cP_\varphi = (P_\varphi, \widetilde{\cF}^s, \widetilde{\cF}^u)$ be the orbit space of $\varphi$. Then $\varphi$ is said to be \emph{$\R$-covered} if both leaf spaces $\Lambda(\widetilde{\cF}^s)$ and $\Lambda(\widetilde{\cF}^u)$ are Hausdorff, i.e., homeomorphic to $\R$.
\end{definition}

In fact, the work of Barbot \cite{barbot_1995} and Fenley \cite{fenley1994_anosov_flow} showed that for Anosov flows, the condition for being $\R$-covered can be relaxed to having just one leaf space homeomorphic to $\R$, because it implies that the other leaf space is also homeomorphic to $\R$. Their work also showed a trichotomy for the structure of the orbit space of an Anosov flow on a 3-manifold, which can be extended to the Anosov-like setting:

\begin{theorem}[\cite{BFM2022orbit}, Theorem 2.16]\label{theorem: trichotomy}
    Let $G \acts \cP = (P, \cF^+, \cF^-)$ be an Anosov-like action on a bifoliated plane. Then exactly one of the following holds:
    \begin{enumerate}[label = (\roman*)]
        \item $\cP$ is isomorphic\footnote{Here, by isomorphism we mean a foliation-preserving homeomorphism.} to $\R^2$ foliated by horizontal and vertical lines, in which case $\cP$ is called the \emph{trivial} plane.
        \item $\cP$ is isomorphic to the strip $\{(x, y) \in \R^2 \mid x < y < x + 1\}$ foliated by horizontal and vertical lines, in which case $\cP$ is called a \emph{skew} plane.
        \item Either the foliations are singular, or both leaf spaces $\Lambda(\cF^+)$ and $\Lambda(\cF^-)$ are non-Hausdorff.
    \end{enumerate}
\end{theorem}

When $\cP = \orbitspace$, a result originally due to Solodov, is that the orbit space of $\varphi$ is trivial if and only if $\varphi$ is orbit equivalent to a suspension flow (see, e.g., \cite[Theorem 2.3.3]{BM_book}). When the leaf spaces are non-Hausdorff, $\cP$ contains \emph{nonseparated} leaves, which are leaves whose projections in their respective leaf spaces are nonseparated as points. We will call a leaf that is nonseparated from some other leaves a \emph{branching leaf}. 

Combining \ref{anosov-like axiom hyperbolic} and \ref{Anosov-like axiom branching} of an Anosov-like action, one observes that every branching leaf in $\cP$ contains a point fixed by some nontrivial element of $G$. In the particular case where $\cP = \orbitspace$, it follows that every branching leaf in $\cF^s$ or $\cF^u$ contains a periodic orbit of $\varphi$.

Two rays in a bifoliated plane are said to make a \emph{perfect fit} if they ``almost intersect'' in the sense that one is contained in a boundary leaf of the saturation of the other and vice-versa. 

\begin{figure}[h]
    \centering
    \subfigure[]{\includegraphics[width = 0.3\textwidth]{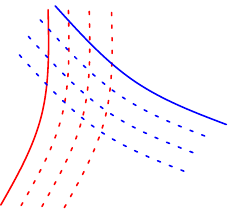}}
    \hspace{1.5cm}
    \subfigure[]{\includegraphics[width = 0.35\textwidth]{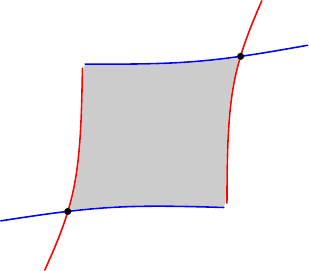}}
    \caption{Two solid leaves making a perfect fit (a); and a lozenge (b)}
    \label{fig: perfect_fit_lozenge}
\end{figure}

\begin{definition}\label{def: lozenge}
    Let $x, y \in \cP = (P, \cF^+, \cF^-)$ such that there exists a ray $A$ of $\cF^+(x)$ making a perfect fit with a ray $D$ of $\cF^-(y)$, as does a ray $B$ of $\cF^-(x)$ with a ray $C$ of $\cF^+(y)$. Then the open region
    \[
    L \coloneqq \{p \in \cP \mid \cF^+(p) \cap B \neq \varnothing, \cF^-(p) \cap A \neq \varnothing\}
    \]
    is called a \emph{lozenge} with \emph{corners} $x, y$ and \emph{sides} $A, B, C, D$.

    The union of a lozenge with all of its corners and sides is called a \emph{closed lozenge}.

    A \emph{chain of lozenges} is a union of closed lozenges such that for any two lozenges $L$ and $L^\prime$ in the chain, there exists lozenges $L_0, L_1, \dots, L_k$ in the chain such that $L_0 = L$, $L_k = L^\prime$, and for all $i \in \{0, 1, \dots k - 1\}$ the pair of lozenges $L_i$ and $L_{i+1}$ share a corner.
\end{definition}

In the definition of a chain of lozenges, we note that lozenges $L_i$ and $L_{i+1}$ may also share a side, but are not required to.

We recall the following important result regarding the relationship between fixed points of an Anosov-like action and corners of lozenges in a bifoliated plane.

\begin{proposition}[\cite{BFM2022orbit}, \cite{Fenley_homotopic_properties}]\label{prop: free homotopy}
Let $G \acts \PP = (P, \cF^+, \cF^-)$ be an Anosov-like action on a bifoliated plane. If some nontrivial element $g \in G$ has distinct fixed points $x, y \in \PP$, then $x$ and $y$ are corners of a chain of lozenges, all corners of which are fixed by $g$.

In particular, if $G \acts \cP$ is the induced action of an Anosov flow $\varphi$, i.e., $G = \pi_1(M)$ and $\cP = \orbitspace$, then $x = \widetilde{\alpha}, y = \widetilde{\beta}$ are the projections of lifts of periodic orbits $\alpha, \beta$ of $\varphi$, where $\alpha, \beta$ are freely homotopic to $g$ as unoriented curves.
\end{proposition}

\section{Homological conditions for transitivity} \label{section: rfrs covers}

In the next two subsections, we will prove Theorem \ref{thmintro: Y-full} and deduce (a more general version of) Theorem \ref{thmintro: ubiquitous existence}. We start by recalling the notions of covers we are using in this article.

\begin{definition} \label{def: covers}
    A \emph{cover} $\overbar{M}$ of a manifold $M$ is a quotient $\overbar{M} = \widetilde{M}/H$ where $H$ is a subgroup of $\pi_1(M)$. Let $G = G(\overbar{M})$ be the group acting on $\overbar{M}$ by deck transformations. The cover is \emph{regular} if $H$ is a normal subgroup of $\pi_1(M)$, in which case we have a natural identification $G\cong \pi_1(M)/H$ as well as a base-point independent identification $H\cong \pi_1(\overbar{M})$. The cover is called \emph{abelian} if $G$ is an abelian group.

    The \emph{maximal abelian cover} of $M$, denoted $M_\mathrm{ab}$, is the regular abelian cover whose fundamental group $\pi_1(M_\mathrm{ab})\cong [\pi_1(M), \pi_1(M)]$ is the commutator subgroup of $\pi_1(M)$, and the associated group of deck transformations is $H_1(M;\Z)\cong \pi_1(M) / [\pi_1(M), \pi_1(M)]$. 
\end{definition}

Note that given any regular abelian cover $\overbar{M}$ with associated group of deck transformations $G\cong \pi_1(M)/\pi_1(\overbar{M})$, then the fact that $G$ is abelian implies that $[\pi_1(M), \pi_1(M)]$ must be a subgroup of $\pi_1(\overbar{M})$, and therefore $\overbar{M}$ is covered by $M_\mathrm{ab}$, hence the name. In particular, the deck transformation group $G$ of such a regular abelian cover can always be identified with a quotient of $H_1(M;\Z)\cong \pi_1(M) / [\pi_1(M), \pi_1(M)]$.




\subsection{Y-full flows}

Throughout this subsection, we assume that $Y$ is an \emph{abelian} regular cover of $M$, and $G$ is the (abelian) group of covering transformations of $Y$. In particular, $G \cong \pi_1(M) / \pi_1(Y)$. 

Given a periodic orbit $\gamma$ of $\flow$ and an element $g\in G$, we will say that \emph{$g$ is represented by $\gamma$} if a lift $\gamma_Y$ of $\gamma$ to $Y$ is left invariant by $g$, and such that for all $x_Y\in \gamma_Y$, $g \cdot x_Y= \flow_Y^T(x_Y)$ for some fixed \emph{positive} $T$. Equivalently, $g$ is represented by $\gamma$ if and only if there exists an element $h\in \pi_1(M)$ such that $g= h\pi_1(Y)$ (for the isomorphism $G \cong \pi_1(M) / \pi_1(Y)$) and $h$ is freely homotopic to $\gamma$ as \emph{oriented} curves.
Note that a lift $\gamma_Y$ is a periodic orbit in $Y$ if and only if $\gamma$ is represented by the identity in $G$.

Recall the following definition:
\begin{definition}[Dougall--Sharp \cite{DS21}]\label{def: Y-full}
    Let $\varphi \colon M \to M$ be an Anosov flow, and let $Y$ be a regular abelian cover of $M$ with group of deck transformations $G$. The flow $\varphi$ is said to be $\emph{Y-full}$ if every class in $G$ is represented by a periodic orbit of $\varphi$.
\end{definition}

In the particular case where $Y= M_{\mathrm{ab}}$ is the maximal abelian cover of $M$, then being $M_{\mathrm{ab}}$-full is called being \emph{homologically full}; see \cite{Sharp_1993}.
As we recalled in the introduction, Theorem \ref{thmintro: Y-full} is an extension of a result of Gogolev and Rodriguez Hertz in \cite{gogolev2020abelianlivshitstheoremsgeometric} where they proved that a transitive Anosov flow is homologically full if and only if its lift to the maximal abelian cover is transitive. 
Here we adapt and extend their argument to deal with other abelian covers.

To prove Theorem \ref{thmintro: Y-full}, it is enough to only consider non-torsion elements in $G$. 
For simplicity we introduce the following:
\begin{notation}
   For a group $G$, we denote by $G_T$ the torsion subgroup of $G$, i.e., $G_T = \tor(G)$, and by $\overbar{G}_T$ the quotient group $G / G_T$.  We also write $H_1(M)$ for $H_1(M; \Z)$.  In particular, $\overbar{H_1(M)}_T = H_1(M) / H_1(M)_T = H_1(M; \Z) / \tor(H_1(M; \Z))$. 
\end{notation}

Recall from the beginning of this section that $G$ can always be identified with some quotient of $H_1(M)$, thus $\overbar{G}_T$ is also identified with some quotient of $\overbar{H_1(M)}_T$.



Now we complete the proof of Theorem \ref{thmintro: Y-full}.

\begin{proof}[Proof of Theorem \ref{thmintro: Y-full}]
    The fact that if the lifted flow on $Y$ is transitive then $\flow$ is $Y$-full was proved in \cite[Lemma 7.4]{DS21}. So we only prove the other direction. 
    
    Suppose that $\varphi$ is $Y$-full. Our goal is to show that the lifted flow $\flow_Y \colon Y \to Y$ has a dense set of periodic orbits, which then by a classical result (for example, see \cite[Theorem 5.3.50]{fisher_hasselblatt_hyperbolic_flows}) implies that $\flow_Y$ is transitive.

    Since $\overbar{G}_T \cong \overbar{H_1(M)}_T / K$, where $K$ is the kernel of some surjective map from $\overbar{H_1(M)}_T$ to $\overbar{G}_T$, we can consider any class $[gK]$ in $\overbar{G}_T$ as a class $[g]$ in $\overbar{H_1(M)}_T$.

    Fix a small $\varepsilon > 0$, and let $0<\delta<\varepsilon$ be a constant for which the Anosov closing lemma applies (see e.g., \cite[Theorem 5.3.10]{fisher_hasselblatt_hyperbolic_flows}), i.e., if an orbit segment of $\flow$ comes back at distance at most $\delta$ from its starting point, then there exists a periodic orbit that $\varepsilon$-shadows it.
    
    Now let $\gamma$ be an $\delta$-dense periodic orbit of $\varphi$. If $[\gamma]$ is trivial in $\overbar{G}_T$, then lifts of (a power of) $\gamma$ to $Y$ are closed orbits and their union form an $\delta$-dense family of close orbits of $\varphi_Y$. So for the rest of this proof, suppose that $[\gamma]$ is nontrivial in $\overbar{G}_T$. By the assumption that $\varphi$ is $Y$-full, we know there exists a periodic orbit $\gamma^\prime$ of $\varphi$ such that $[\gamma] + [\gamma^\prime]$ is trivial in $\overbar{G}_T$. Consider $[\gamma]$ and $[\gamma^\prime]$ as elements of $\overbar{H_1(M)}_T$. 
    
    Since $\gamma$ is $\delta$-dense, for any $y \in \gamma^\prime$, there exists some $x \in \gamma$ such that $y$ is contained in the $\delta$-neighborhood of $x$. By the Anosov closing lemma, there exists a periodic orbit $\eta$ of $\flow$ in the $\varepsilon$-neighborhood of the pseudo-orbit obtained from concatenating $\gamma$ and $\gamma^\prime$ through $x$ and $y$. It follows that $\eta$ is $\varepsilon$-dense in $M$. Using Fried's ``pair of pants'' construction (see \cite[\S 2]{fried1983transitive}), we can construct an immersed surface whose boundary consists of $\gamma, \gamma^\prime$, and $\eta$. Then $\eta$ is homologous to $\gamma + \gamma_1$, so $[\eta] = [\gamma] + [\gamma^\prime]$ is trivial in $\overbar{G}_T$. Then $\eta$ lifts to an $\varepsilon$-dense union of closed orbits of $\varphi_Y$ in $Y$. This proves that $\varphi_Y$ is transitive. 
\end{proof}

The technique used in the proof of Theorem \ref{thmintro: Y-full} can also be used to give an easily verifiable condition to build abelian covers $Y$ for which $\flow$ is $Y$-full.

\begin{theorem}\label{theorem: construct Y-full cover}
    Let $\flow \colon M \to M$ be a transitive Anosov flow. Suppose that there exists linearly independent nontrivial elements $[\alpha_1], \dotsc, [\alpha_n] \in \overbar{H_1(M)}_T$ such that for each $i = 1, \dots, n$, both $[\alpha_i]$ and $-[\alpha_i]$ are represented by periodic orbits of $\varphi$. Let $A = \mathrm{span}(\{[\alpha_1], \dotsc, [\alpha_n]\})$, and write $\overbar{H_1(M)}_T = A \oplus K$ for some normal subgroup $K \lhd \overbar{H_1(M)}_T$. Let $Y$ be the abelian cover of $M$ associated to $K$, i.e., the cover whose group of deck transformations $G \cong \overbar{H_1(M)}_T / K$. Then the lifted flow $\flow_Y \colon Y \to Y$ is transitive, and in particular, $\flow$ is $Y$-full.
\end{theorem}

\begin{rem}
    This theorem tells us that as soon as there exists some nontrivial $[\alpha] \in \overbar{H_1(M)}_T$ such that both $[\alpha]$ and $-[\alpha]$ are represented by some periodic orbits of $\varphi$, then for any infinite cyclic cover $Y$ of $M$ associated to an element $[\delta] \in H^1(M; \Z)$ with $[\delta]([\alpha]) \neq 0$, the lifted flow $\flow_Y$ is transitive.
\end{rem}

\begin{proof}[Proof of Theorem \ref{theorem: construct Y-full cover}]
    Fix some small $\varepsilon > 0$, let $\eta_0$ be an $\varepsilon$-dense periodic orbit of $\varphi$. Then we can write its homology class $[\eta_0] = \sum^n_{i = 1} c_i [\alpha_i] + [\beta]$, where $[\beta] \in K$ and $c_i\in \Z$ for all $i$. If $c_i = 0$ for all $i$, then $[\eta_0] = [\beta] \in K$, so $\eta_0$ lifts to an $\varepsilon$-dense union of closed orbits in $Y$, implying that the lifted flow $\flow_Y$ is transitive. So for the rest of this proof, up to renaming, we assume that for some positive integer $k \leq n$, $c_1, \dots, c_k \neq 0$.

    Suppose that $c_1 > 0$. By assumption, there exists a periodic orbit that represents the homology class $-[\alpha_1]$. By going around that orbit $c_1$-times, we obtain a periodic orbit $\gamma_1$ that represents the homology class $-c_1[\alpha_1]$. As in the proof of Theorem \ref{thmintro: Y-full}, by applying the Anosov closing lemma and then Fried's ``pair of pants'' construction, we find an $\varepsilon$-dense periodic $\eta_1$ of $\flow$ such that $\eta_1$ is homologous to $\eta_0 + \gamma_1$. Then
    \[
        [\eta_1] = [\eta_0] + [\gamma_1] = \sum^k_{i = 1} c_i [\alpha_i] + [\beta] - c_1[\alpha_1] = \sum^k_{i = 2} c_i [\alpha_i] + [\beta].
    \]
    If $c_1 < 0$, then by assumption, there exists a periodic orbit that represents the homology class $[\alpha]$, and we obtain $\gamma_1$ by going over that orbit $-c_1$-times.

    We then repeat this process for $i = 2, \dots, k$.
    For each $i$, we obtain an $\varepsilon$-dense periodic orbit $\eta_i$ of $\flow$ such that 
    \[
    [\eta_i] = \sum^k_{i_0 = i+1} c_{i_0} [\alpha_{i_0}] + [\beta].
    \]
    In the end, we obtained an $\varepsilon$-dense periodic orbit $\eta_k$ of $\flow$ such that $[\eta_k] = [\beta] \in K$. Then $\eta_k$ lifts to an $\varepsilon$-dense union of closed orbits of $\flow_Y$, implying that $\flow_Y$ is transitive, and, by Theorem \ref{thmintro: Y-full}, that $\flow$ is $Y$-full.    
\end{proof}

\subsection{Manifolds with RFRS fundamental groups and their covers}

We now wish to show that many Anosov flows satisfy the hypothesis of Theorem \ref{theorem: construct Y-full cover}, in order to prove Theorem \ref{thmintro: ubiquitous existence}. We will in fact prove a more general result, see Theorem \ref{theorem: rfrs cover} below.

First, recall that Theorem \ref{thm: homotopical charachterization suspension} implies that, up to taking a double cover, for \emph{any} Anosov flow which is not orbit equivalent to a suspension, there exists a pair of distinct periodic orbits $\alpha$, $\alpha'$ such that $\alpha$ is freely homotopic to the inverse of $\alpha'$. This implies that $[\alpha]$ and $[\alpha']=-[\alpha]$ are both represented by periodic orbits. The only issue preventing us from directly applying Theorem \ref{theorem: construct Y-full cover} is that we may in general have $[\alpha] = 0 \in \overbar{H_1(M)}_T$.

Deep results in $3$-manifolds topology, due to work of Agol and Wise, ensures that for a lot of $3$-manifolds, one can always take finite covers to ensure that $[\alpha]$ is not a torsion element in $H_1(M)$. This property, which was introduced by Agol in \cite{Agol_2008}, is that of \emph{residually finite rationally solvable (RFRS)} groups. We start by recalling its definition:

\begin{definition}[Agol] \label{def: rfrs}
    A group $G$ is \emph{RFRS} if there is a sequence of subgroups $G = G_0 > G_1 > G_2 > \cdots$ such that $G_i$ is a normal subgroup of $G$ for all $i$, $\bigcap_i G_i = \{1_G\}$, $[G \colon G_i] < \infty$, and $G_{i+1} \leq (G_i)^{(1)}_r$, where $(G_i)^{(1)}_r = \{x \in G \mid \exists k \neq 0, x^k \in [G_i, G_i]\}$.

    We call a sequence of subgroups of $G$ satisfying the condition above an \emph{RFRS} sequence of $G$.
\end{definition}

Note that any subgroup of an RFRS group is also RFRS. The following fundamental result shows how common RFRS groups are among $3$-manifold groups.

\begin{theorem}[Agol \cite{agol_virtual_haken_conjecture},  Przytycki--Wise \cite{wise_mixed_3_manifolds}, Liu \cite{Liu_2013}]\label{thm: virtually special}
    Let $M$ be an orientable, irreducible $3$-manifold. Suppose that $M$ is \emph{not} a graph manifold without a non-positively curved metric, then $\pi_1(M)$ is virtually special.\footnote{In our context, one can define a group to be \emph{special} if it is a subgroup of a right-angled Artin group, see e.g., \cite{wise_mixed_3_manifolds} or \cite[Corollary 5.8]{3_manifold_groups_book} for more precisions.} In particular, $\pi_1(M)$ is virtually RFRS.
\end{theorem}
\begin{proof}
    The fact that the fundamental group of hyperbolic $3$-manifolds is virtually special is due to Agol \cite{agol_virtual_haken_conjecture}. Przytycki--Wise \cite{wise_mixed_3_manifolds} proves it in the case of manifolds whose JSJ-decomposition is not trivial and contains at least one atoroidal piece. Finally, Liu \cite{Liu_2013} treats the case of graph manifolds that admit a non-positively curved metric.
    One of the consequence of being a special group is to be RFRS, see, e.g., \cite[Corollary 5.8]{3_manifold_groups_book} and \cite[Theorem 2.2]{Agol_2008}.
\end{proof}

To prove Theorem \ref{thmintro: ubiquitous existence}, we will need the following technical lemma. For a group $G$, we denote by $[G, G]$ its commutator subgroup. and we denote by $H_1(G) = H_1(G; \Z)$ the quotient group $G/ [G, G]$. 

\begin{lemma}\label{lemma: rfrs lemma}
    Let $G$ be an RFRS group. Then for any cyclic subgroup $\langle g \rangle$ (finite or infinite) of $G$, there exists a finite-index subgroup $G^\prime$ such that $g \in G^\prime$ and such that $g$ represents a nontrivial element $[g]$ in the torsion-free abelian group $\overbar{H_1(G^\prime)}_T$.
\end{lemma}

\begin{proof}
    To prove this lemma, we use the following fact.

    \begin{claim}\label{claim: rfrs factor through}
        For all $i$, the map $G_i \to G_i / G_{i+1}$ factors through the map $G_i \to \overbar{H_1(G_i)}_T$. In other words, the following is a commutative diagram
        \begin{center}
            \begin{tikzcd}
                G_i \arrow[dr] \arrow[rr]
                    & & G_i / G_{i+1}  \\
                    & \overbar{H_1(G_i)}_T \arrow[ur]
        \end{tikzcd}
        \end{center}
    \end{claim}
The proof of this claim is an exercise in group theory, it is even taken as a replacement of the condition $G_{i+1} \leq (G_i)^{(1)}_r$ in the definition of RFRS given in \cite{3_manifold_groups_book}. We provide a proof for completeness.

    \begin{proof}[Proof of Claim \ref{claim: rfrs factor through}]
        For any $i$, we have
        \begin{align*}
            \overbar{H_1(G_i)}_T &= \frac{G_i/[G_i, G_i]}{\tor(G_i/[G_i, G_i])} \\
            &= \frac{G_i/[G_i, G_i]}{\{x \in G_i \mid \exists k \neq 0, x^k \in [G_i, G_i]\}} \\
            &= \frac{G_i/[G_i, G_i]}{(G_i)^{(1)}_r} \\
            &= G_i / (G_i)^{(1)}_r.  
        \end{align*}
        The last equality follows from the fact that $[G_i, G_i] < (G_i)^{(1)}_r$. \par

        On the other hand, by the definition of RFRS groups, we have $(G_i)^{(1)}_r \lhd G_{i+1} \lhd G_i$. Then by the Third Isomorphism Theorem, we have
        \[
        G_i/G_{i+1} \cong \frac{G_i/(G_i)^{(1)}_r}{G_{i+1}/(G_i)^{(1)}_r}.
        \]
        Then we are able to build a map 
        \begin{align*}
            G_i \longrightarrow \overbar{H_1(G_i)}_T &\cong G_i / (G_i)^{(1)}_r \longrightarrow \frac{G_i/(G_i)^{(1)}_r}{G_{i+1}/(G_i)^{(1)}_r},
        \end{align*}
        which induces the commutative diagram
        \begin{center}
            \begin{tikzcd}
            g \arrow[rr, mapsto] \arrow[dr, mapsto]
                & & g(G_i)^{(1)}_r \\
                & g(G_i)^{(1)}_r\left(G_{i+1}/(G_i)^{(1)}_r\right) \arrow[ur, mapsto].
        \end{tikzcd}
        \end{center}
    \end{proof}

    Now, going back to the proof of the lemma, consider the cyclic subgroup $\langle g \rangle < G$. Let $G^\prime = G_i$, where $G_i$ is the subgroup in the RFRS sequence of $G$ such that $g \in G_i$ and $g \notin G_{i+1}$. Equivalently, we have that $\langle g \rangle < G_i$ and $\langle g \rangle \nless G_{i+1}$. Then the map $G_i \to G_i/G_{i+1}$ takes $g$ to a nontrivial element in $G_i/G_{i+1}$, and it follows from Claim \ref{claim: rfrs factor through} that $g$ represents a nontrivial element in $\overbar{H_1(G^\prime)}_T$.
\end{proof}

We are now ready to state and prove the main result of this section, from which we will deduce Theorem \ref{thmintro: ubiquitous existence}.

\begin{theorem}\label{theorem: rfrs cover}
    Let $\varphi$ be a transitive Anosov flow on $M$ whose fundamental group is virtually RFRS. Assume that $\flow$ is not a suspension Anosov flow. Then there exists a regular infinite cover $\overbar{M}$ of $M$ such that the lifted flow $\overbar{\varphi}$ on $\overbar{M}$ is transitive. 
    
    Moreover, we can choose the cover $\overbar M$
    to be such that the action of $\pi_1(\overbar{M})$ on $\cP_\varphi = \orbitspace$ is topologically transitive and does not fix any branching leaves.
\end{theorem}

The second part of the above result is particularly interesting in the context of Anosov-like actions, as it shows that Axiom \ref{Anosov-like axiom branching} of Definition \ref{def: anosov-like} may indeed fail for the action induced by Anosov flows on non-compact manifolds, even when all the other axioms are satisfied. We will further build upon this in Section \ref{section: non-compact suspension like}.

\begin{rem}
The cover $\overbar{M}$ we build is in general not an abelian cover of $M$, but is an abelian cover of a finite cover of $M$. Note also that the conclusion of this theorem holds for the case of suspension Anosov flows, but the argument for the proof is completely different. In fact, a far stronger conclusion holds for both suspension and skew Anosov flows, see Theorem \ref{theorem: R-covered transitivity}.
\end{rem}

\begin{proof}
    We will build the desired cover $\overbar{M}$ explicitly. We will first find a finite cover $\hat{M}$ of $M$ such that all periodic orbits of $\varphi$ on branching leaves --- if there are any --- represent non-torsion elements in $H_1(\hat{M})$. Then we will construct $\overbar{M}$ by applying Theorem \ref{theorem: construct Y-full cover}. \par

First, up to taking a cover of degree at most $4$, we can assume that $M$ is orientable and $\flow$ is transversely orientable (i.e., both its stable and unstable foliations are transversely orientable). Second, up to taking a further cover if necessary, we may assume that $\pi_1(M)$ is RFRS.

    By Theorem F of \cite{fenley98_structure_of_branching}, we know that $\varphi$ has finitely many periodic orbits on branching leaves. These periodic orbits thus determine finitely many distinct (unoriented) free homotopy classes $g_1, \dots, g_n\in \pi_1(M)$. This list is empty if and only if the flow $\flow$ is $\R$-covered, which, by Theorem \ref{theorem: trichotomy} and our assumption that $\flow$ is not a suspension, implies that $\flow$ must be skew. In that case, we take $g_1$ to be any element of $\pi_1(M)$ that is represented by a periodic orbit of $\flow$ instead.

    With this choice, we find that for each $i$, $g_i$ leaves invariant a lozenge in $\cP_\flow$. Thus, both $g_i$ and $g_i^{-1}$ are represented by periodic orbits of $\flow$.

    Our first step will consist of showing that one can take a regular finite cover $\hat M$ of $M$ such that the lifts of (powers of) all the $g_i$ represent non-torsion elements in $H_1(\hat M)$.

    Consider the subcollection $\{g_{n_1}, \dots, g_{n_k}\}$ of $\{g_1, \dots, g_n\}$ such that $[g_{n_i}] = 0 \in \overbar{H_1(M)}_T$. Up to renaming the $g_i$'s, we assume that $g_1, \dots, g_k$ with $k \leq n$ are torsion elements in $H_1(M)$. Fix an RFRS sequence for $\pi_1(M)$. By Lemma \ref{lemma: rfrs lemma}, we know that for each $g_i$, there exists a finite-index normal subgroup $G_i$ of $\pi_1(M)$ in the fixed RFRS sequence such that $[g_i]$ is nontrivial in $\overbar{H_1(G_i)}_T$. Up to renaming $\{g_1, g_2, \dots, g_k\}$ again, we assume that $G_1$ is the largest subgroup among $\{G_i\}_{1 \leq i \leq k}$, i.e., $G_1$ is the first subgroup appearing in the fixed RFRS sequence of $\pi_1(M)$, so that $g_1, \dots, g_n \in G_1$. Let $M_1$ be the finite cover of $M$ corresponding to $G_1$, i.e., $G_1 = \pi_1(M_1) < \pi_1(M)$  (recall that by definition of RFRS, $G_1\triangleleft \pi_1(M)$ so $M_1$ is a regular cover). Then $[g_1]$ is nontrivial in $\overbar{H_1(M_1)}_T$. Note that if $k < n$, i.e., there exists some $g_i$ such that $[g_i] \neq 0 \in \overbar{H_1(M)}_T$, then since $M_1$ is a finite cover of $M$, up to replacing $g_i$ by some finite power of it, we have $g_i \in \pi_1(M_1)$ and $[g_i] \neq 0 \in \overbar{H_1(M)}_T$. \par

    We also note that it is possible to have other $[g_i]$ to be nontrivial in $\overbar{H_1(M)}_T$. This does not pose any issue, as we are able to produce a strictly smaller collection of potentially trivial classes of $\overbar{H_1(M)}_T$ represented by periodic orbits of $\flow$. \par
        
    Now, since subgroups of an RFRS group are also RFRS, we can repeat the process for $\pi_1(M_1)$. More precisely, $g_1, \dots, g_n$ are all elements of $\pi_1(M_1)$, and at least $g_1$ is such that $[g_1]\neq 0 \in \overbar{H_1(M)}_T$. Now consider the subcollection $\{g_{n_1},\dots, g_{n_l}\}$ of $\{g_2, \dots, g_n\}$ such that $[g_{n_j}] = 0\in \overbar{H_1(M)}_T$. Up to renaming the $g_i$, we may assume as above that $g_2, \dots, g_l$, with $l\leq n$ are torsion elements in $H_1(M_1)$. As above, we fix an RFRS sequence for $M_1$ to get a subgroup $G_2$ (which is finite index in $G_1$) such that $g_2, \dots, g_l\in G_2$, and, up to further renaming, $[g_2]$ is nontrivial in $\overbar{H_1(G_2)}_T$. Then we consider $M_2$ the finite regular cover of $M_1$ such that $\pi_1(M_2)=G_2$. By construction, in $M_2$, $[g_2]$ represents a non-torsion element of $H_1(M_2)$. Moreover, since $M_2$ is a finite cover of $M_1$ and $[g_1]\neq 0 \in \overbar{H_1(M_1)}_T$, we deduce that, up to replacing $g_1$ by some finite power of it so that $g_1\in \pi_1(M_2)$, we have $[g_1]\neq 0 \in \overbar{H_1(M_2)}_T$.

    This inductive process terminates in at most $n$ steps, and thus we obtained a finite regular cover $\hat M$ of $M$, such that appropriate powers of $g_1,\dots, g_n$ all represent non-torsion elements in $H_1(\hat M)$.

    Now we have a finite collection of nontrivial elements $[g_1], [g_2], \dots, [g_n] \in \overbar{H_1(\hat{M})}_T$. By construction each $[g_i]$ and $[g_i^{-1}]=-[g_i]$ are represented by periodic orbits of the lifted flow $\hat\flow$ on $\hat M$. Moreover, since $\hat{M}$ is a finite cover of $M$, $\hat{\flow}$ is transitive. Then by applying Theorem \ref{theorem: construct Y-full cover}, we obtain an abelian cover $\overbar{M}$ of $\hat M$ such that the flow $\overbar{\flow}$ obtained by lifting $\hat \flow$ is transitive.

    To prove the last statement of the theorem, notice that, by construction, $\pi_1(\overbar{M})$ does not contain any element conjugated to a power of one of the elements $g_i$. Now, the elements $g_i$ were chosen so that any branching leaf in the orbit space of $\flow$ (which is the same as the orbit space of $\bar\flow$) was fixed by a conjugate of (any power of) one of the $g_i$. Hence, no element in $\pi_1(\overbar{M})$ fixes any branching leaf.
\end{proof}

As mentioned in Theorem \ref{thm: virtually special}, most orientable irreducible $3$-manifolds have virtually RFRS fundamental groups. Since any $3$-manifold admitting an Anosov flow  is irreducible, we immediately deduce the following, which in particular gives Theorem \ref{thmintro: ubiquitous existence} from the introduction:

\begin{corollary}\label{corollary: rfrs manifolds}
    Let $M$ be a compact 3-manifold that admits a transitive Anosov flow $\varphi$. If $M$ is either a hyperbolic manifold, or a mixed manifold, or a non-positively curved graph manifold, then there exists an infinite regular cover $\overbar{M}$ of $M$ on which the lifted flow $\overbar{\varphi}$ is transitive.
\end{corollary}

\begin{rem}
    It is very likely that for any transitive Anosov flow on a graph manifold $M$ that is not a suspension, one can build a finite cover such that the condition of Theorem \ref{theorem: construct Y-full cover} is satisfied for at least one non-torsion element in $H_1(M)$. For instance, if $M$ supports a transitive Anosov flow and there exists an essential torus $\T$ in $M$ such that the map $H_1(\T)\to H_1(M)$ is injective, then one can use Theorem \ref{theorem: construct Y-full cover} to build an infinite cyclic cover $Y$ on which the lifted flow will be transitive. Indeed, for any essential torus $\T$ in a manifold supporting an Anosov flow that is not a suspension, there always exists an element $[\alpha]\in H_1(\T)$ such that both $[\alpha]$ and $-[\alpha]$ are represented by periodic orbits of the flow; see, e.g., \cite[Theorem 6.10]{BF13}.
\end{rem}

\section{$\R$-covered flows and regular covers.} \label{section: regular covers}

Recall that an Anosov flow is said to be \emph{$\R$-covered} if the leaf spaces $\Lambda(\cF^s)$ and $\Lambda(\cF^u)$ are homeomorphic to $\R$.

The point of this section is to prove Theorem \ref{thmintro: R-covered} that we restate:

\begin{theorem}\label{theorem: R-covered transitivity}
    Let $\varphi \colon M \to M$ be a $\R$-covered Anosov flow. Let $\overbar{M}$ be a regular cover of $M$ and $\overbar{\varphi}$ be the lift of $\varphi$ to $\overbar{M}$. Then $\overbar{\varphi}$ is transitive if and only if $\overbar{\varphi}$ admits a periodic orbit, and if and only if it has a dense set of periodic orbits.
    
    Moreover, if $\overbar{\varphi}$ is not transitive, then we have exactly one of the three following cases:
    \begin{enumerate}
        \item $\overbar{M} = \widetilde{M}$, i.e., it is the universal cover;
        \item $\overbar{M}$ is a fiberwise cover of $T^1\widetilde{\Sigma}$, where $\widetilde{\Sigma} \cong \mathbb{H}^2$ is the universal cover of some hyperbolic orbifold $\Sigma$, and $\varphi$ is a lift of the geodesic flow on $T^1\Sigma$;
        \item $\overbar{M} = \mathbb{T}^2 \times \R$ or $\overbar{M} = \mathbb{S}^1 \times \R \times \R$, and $\varphi$ is the suspension flow of an Anosov diffeomorphism on a torus.
    \end{enumerate}
\end{theorem}

Note that we are now considering all regular covers and not just abelian ones.

In \cite{BBM24}, the authors introduced the notion of a \emph{Smale class} of an Anosov-like action on a bifoliated plane, which generalizes the notion of basic sets, in the sense of \cite{smale1967differentiable}, of Anosov flows. In particular, \cite[Theorem 1.8]{BBM24} gives a characterization of topologically transitive Anosov-like actions as those that have a unique Smale class (see also \cite[Theorem 2.9.2]{BM_book}). Moreover, as in the compact Anosov flow case, where basic sets are separated by transverse tori (see \cite{Bru}), distinct Smale classes are ``separated'' by some special chains of lozenges called \emph{Smale chains}:

\begin{definition}[Barthelmé--Bonatti--Mann, \cite{BBM24}]\label{def: smale chain}
    Let $G \acts \cP$ be an Anosov-like action. A \emph{Smale chain} for $G \acts \cP$ is a bi-infinite chain of lozenges $\mathcal{W} = \{L_i\}_{i \in \Z}$ such that for all $i$, 
    \begin{enumerate}[label = (\roman*)]
        \item $L_i$ and $L_{i+1}$ share a side. \label{smale_chain_share_side}
        \item no three consecutive lozenges $L_i$, $L_{i+1}$ and $L_{i+2}$ share a common corner. \label{Smale_chain_no_shared_corner} 
        \item\label{smale_chain_wandering} each lozenge $L_i$ is wandering, that is, no points of the interior of $L_i$ are fixed by any nontrivial element of $G$.
    \end{enumerate}
\end{definition}

\begin{figure}[h]
    \centering
    \includegraphics[width=0.5\linewidth]{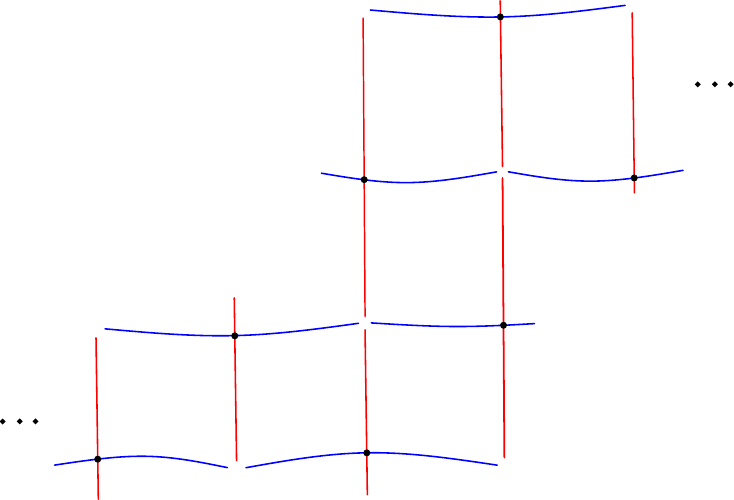}
    \caption{A local leaf configuration of a Smale chain}
    \label{fig: smale chain}
\end{figure}

Theorem 1.19 of \cite{BBM24} gives that if an Anosov-like action is not topologically transitive, then there exists at least one such Smale chain. Equivalently, one can characterize the transitivity of Anosov-like actions as those that do \emph{not} admit any Smale chain. Note that, while condition \ref{smale_chain_wandering} in Definition \ref{def: smale chain} is dynamical in nature, conditions \ref{smale_chain_share_side} and \ref{Smale_chain_no_shared_corner} only involve the configuration of foliations in the plane, and not the group action. Thus, we can deduce the following sufficient condition for \emph{all} Anosov-like actions to be transitive on certain bifoliated planes:
\begin{proposition}\label{prop_sufficient_transitive}
Let $\cP$ be a bifoliated plane. Assume that $\cP$ does not contain any chains of lozenges $\{L_i\}_{i\in \Z}$ such that for all $i$, 
     \begin{enumerate}[label = (\roman*)]
        \item $L_i$ and $L_{i+1}$ share a side. \label{item_chain_share_side}
        \item no three consecutive lozenges $L_i$, $L_{i+1}$ and $L_{i+2}$ share a common corner. \label{item_chain_no_shared_corner}
    \end{enumerate}
Then any Anosov-like action of a group $G$ on $\cP$ is transitive.
\end{proposition}

In particular, the skew and trivial bifoliated planes satisfy the above condition, and so any Anosov-like actions on such planes are automatically transitive. Hence, the proof of the first part of Theorem \ref{theorem: R-covered transitivity}, as well as its generalization in Theorem \ref{thm: transitive cover no infinite chains} below, consists in showing that the fundamental group of the regular cover $\overbar{M}$ still act Anosov-like on the orbit space. Our next lemma proves that this is almost true for any Anosov flow, as soon as the lifted flow admits at least one periodic orbit.

\begin{lemma}\label{lemma: regular cover + periodic orbit is almost Anosov-like}
    Let $\varphi \colon M \to M$ be a transitive Anosov flow on a compact 3-manifold $M$. Let $\overbar{M}$ be a regular cover of $M$ and $\overbar{\varphi}$ be the lift of $\varphi$ to $\overbar{M}$. If $\overbar{\varphi}$ has a periodic orbit, then the action $\pi_1(\overbar{M}) \acts \cP_\varphi$ satisfies Axioms \ref{anosov-like axiom hyperbolic}, \ref{anosov-like axiom dense fixed leaves}, and \ref{anosov-like axiom singularity} of an Anosov-like action.
\end{lemma}

\begin{proof}
    Since $\pi_1(M) \acts \cP_\varphi$ is Anosov-like, $\pi_1(\overbar{M})$ as a subgroup of $\pi_1(M)$, inherits \ref{anosov-like axiom hyperbolic}. Axiom \ref{anosov-like axiom singularity} is vacuously satisfied since $\varphi$ is an Anosov flow and does not have any singular orbit. Finally, we check \ref{anosov-like axiom dense fixed leaves}: By assumption $\overbar{\varphi}$ has a periodic orbit, which is equivalent to having a nontrivial element $h \in \pi_1(\overbar{M})$ such that $h \cdot x = x$ for some $x \in \cP_\varphi$. By item \ref{item_charac_dense_leaves} of Theorem \ref{thm_characterization_transitive}, $\pi_1(M)\cdot \cF^s(x)$ and $\pi_1(M)\cdot \cF^s(x)$ are both dense in $\cP_\flow$. Since $\overbar{M}$ is a regular cover, $\pi_1(\overbar{M})$ is a normal subgroup of $\pi_1(M)$. In particular, for any $g\in \pi_1(M)$, the element $ghg^{-1}$ is in $\pi_1(\overbar{M})$ and fixes the point $g \cdot x$. Thus, each leaf of $\pi_1(M)\cdot \cF^s(x)$ (resp.~$\pi_1(M)\cdot \cF^s(x)$) are fixed by an element of $\pi_1(\overbar{M})$, so Axiom \ref{anosov-like axiom dense fixed leaves} is satisfied.
\end{proof}

Now we prove Theorem \ref{theorem: R-covered transitivity}.

\begin{proof}[Proof of Theorem \ref{theorem: R-covered transitivity}]
    If $\overbar{M}$ is a finite cover then $\overbar{\varphi}$ is clearly transitive. So for the rest of the proof we will assume that $\overbar{M}$ is an infinite cover. \par
    
    We first show that $\overbar{\varphi}$ must be transitive when it has a periodic orbit. By Lemma \ref{lemma: regular cover + periodic orbit is almost Anosov-like}, $\pi_1(\overbar{M}) \acts \cP_\varphi$ satisfies \ref{anosov-like axiom hyperbolic}, \ref{anosov-like axiom dense fixed leaves}, and \ref{anosov-like axiom singularity}. Since $\varphi$ is $\R$-covered, $P_\varphi$ does not have any nonseparated leaves, so \ref{Anosov-like axiom branching} is trivially satisfied. Therefore, $\pi_1(\overbar{M}) \acts P_\varphi$ is an Anosov-like action. Since $\varphi$ is $\R$-covered, $\cP_\varphi$ is trivial or skew, and Proposition \ref{prop_sufficient_transitive} implies that $\pi_1(\overbar{M})$ is transitive, so $\overbar{\varphi}$ is transitive. \par
    
    Conversely, if $\overbar{\varphi}$ is transitive, then $\overbar{\varphi}$ has a dense set of periodic orbits in $\overbar{M}$. \par

    To prove the second part of the statement, one first observes that if $\overbar{M} = \widetilde{M}$ then $\overbar{\varphi}$ does not have any periodic orbit, so $\overbar{\varphi}$ is not transitive. Now, suppose that $\overbar{M} \neq \widetilde{M}$ and $\overbar{\varphi}$ is not transitive. Then, by the above, $\overbar{\varphi}$ does not have any periodic orbit, i.e., $\pi_1(\overbar{M})$ acts freely on $\cP_\varphi$. In particular, $\pi_1(\overbar{M})$ also acts freely on the leaf spaces $\Lambda(\cF^s)$ and $\Lambda(\cF^u)$, each of which are homeomorphic to $\R$. By the classical H\"older's Theorem, we deduce that $\pi_1(\overbar{M})$ is abelian, and hence an abelian normal subgroup of $\pi_1(M)$. By results of Plante \cite{Plante81_transverse_foliation} and Ghys \cite{ghys_1984} (see, e.g., \cite[Theorem 2.11.7]{BM_book} for a proof), $\varphi$ is either a lift of the geodesic flow on the unit tangent bundle of some negatively curved orbifold $\Sigma$ or the suspension flow of an Anosov diffeomorphism. We discuss these further in cases.

    \begin{case}
        Suppose that $\varphi$ is a lift of the geodesic flow on $T^1\Sigma$, so the orbit space of $\varphi$ is skew. Let $h \colon (x, y) \mapsto (x+1, y+1)$ be the \emph{one-step-up map} on the skew plane.
        Then by Theorem 2.11.6 in \cite{BM_book}, $\pi_1(\overbar{M})$ is contained in the cyclic center $\langle h^k \rangle$ of $\pi_1(M)$. The cover associated to $\langle h \rangle$ is $T^1\widetilde{\Sigma}$, and so $\overbar{M}$ is obtained as a (self)-cover $T^1\widetilde{\Sigma}$ by unwrapping the fibers.
    \end{case}

    \begin{case}
        Suppose that $\varphi$ is the suspension flow of an Anosov diffeomorphism $f$ on a torus $\mathbb{T}^2$. Then $\pi_1(M) = \Z^2 \ltimes \Z$. Since $\widetilde{M} \cong \R^3$, $\overbar{M}$ is homeomorphic to $\R^3$ quotiented out by an abelian normal subgroup of $\pi_1(M)$, i.e., a subgroup of the $\Z^2$ factor in $\pi_1(M) = \Z^2 \ltimes \Z$. Since the $\Z^2$ factor acts by integral translations on the lifts of $\T^2$, we have that either $\overbar{M} = \mathbb{T}^2 \times \R \cong \R^3 / \Z^2$ or $\overbar{M} = \mathbb{S}^1 \times \R \times \R \cong \R^3 / \Z$.\qedhere
    \end{case}
\end{proof}

In order to apply Proposition \ref{prop_sufficient_transitive} and the other tools of \cite{BBM24}, to deduce the transitivity on regular covers given by Theorem \ref{theorem: R-covered transitivity}, it was essential to prove that the action of the subgroup of $\pi_1(\overbar{M})$ acted Anosov-like. By Lemma \ref{lemma: regular cover + periodic orbit is almost Anosov-like}, the only additional fact that was needed was Axiom \ref{Anosov-like axiom branching}, i.e., that branching leaves where fixed. While this is vacuously true for $\R$-covered flows, it is very far from true in general, as we have seen in Theorem \ref{theorem: rfrs cover}. Thus, to have a generalization of Theorem \ref{theorem: R-covered transitivity} to the non-$\R$-covered setting, one needs to further assume that Axiom \ref{Anosov-like axiom branching} is satisfied:

\begin{theorem}\label{thm: transitive cover no infinite chains}
    Let $\varphi \colon M \to M$ be a transitive Anosov flow. Let $\overbar{M}$ be an infinite regular cover of $M$ and let $\overbar{\varphi}$ be the lift of $\varphi$ to $\overbar{M}$. Suppose that all branching leaves in $\overbar{M}$ are periodic, and the orbit space $\cP_{\varphi}$ does not contain any infinite chain of lozenges that satisfies conditions \ref{smale_chain_share_side} and \ref{Smale_chain_no_shared_corner} in Definition \ref{def: smale chain}. Then $\overbar{\varphi}$ is transitive.
\end{theorem}

\begin{proof}
    The $\R$-covered suspension case was shown in Theorem \ref{theorem: R-covered transitivity}, while the $\R$-covered skewed case does not satisfy the assumption on lozenges, so we assume that $\varphi$ is not $\R$-covered. \par

 Thanks to Lemma \ref{lemma: regular cover + periodic orbit is almost Anosov-like}, $\pi_1(\overbar{M}) \acts \cP_\varphi$ satisfies Axioms \ref{anosov-like axiom hyperbolic}, \ref{anosov-like axiom dense fixed leaves}, and \ref{anosov-like axiom singularity}. By assumption, it also satisfies Axiom \ref{Anosov-like axiom branching}. Thus $\pi_1(\overbar{M})$ acts Anosov-like on $\cP_\flow$. So Proposition \ref{prop_sufficient_transitive} yields the conclusion.
\end{proof}


\section{Non-compact Anosov flows and homotopical characterization of suspensions}\label{section: non-compact suspension like}

In this section, we prove Theorem \ref{thmintro: noncompact_suspension}, that is, we build an example of a transitive Anosov flow on a non-compact manifold such that every periodic orbit is unique in its free homotopy class.

The strategy for building such examples is easy: Start with a transitive Anosov flow $\flow$. Suppose that $\flow$ is such that it has only finitely many orbits that are not alone in their free homotopy class. If these elements are non-torsion in homology, then we can apply Theorem \ref{theorem: construct Y-full cover} to each of them. If some of these elements are torsion, then, as in the proof of Theorem \ref{theorem: rfrs cover}, and provided that the fundamental group of the manifold is virtually special, we may first take a finite cover to ensure that all these elements are not torsion in homology. Then, the abelian cover obtained by ``unwrapping'' each of these periodic orbits gives a lifted Anosov flow that is transitive and for which any orbit that corresponds to a corner of a lozenge is not periodic.
So, the work left to do is to construct examples of Anosov flows satisfying the property that only finitely many periodic orbits are not alone in their free homotopy class, and either prove by hand that these elements are non-torsion in homology, or make sure that the fundamental group of the manifold is virtually special.

Anosov flows with the property that only finitely many periodic orbits are not alone in their free homotopy class are not uncommon: For instance, all \emph{totally periodic} Anosov flows, in the sense of \cite{BF15}, have that property. For us however, these are not the best examples, as the ambient manifold for such flows is a graph manifold, and so it is harder to justify which of these examples will have the non-torsion homological condition that we need. So instead, we will build an example on a manifold whose JSJ decomposition consists of just one atoroidal piece, using the DA bifurcation construction of Smale \cite{smale1967differentiable} and Williams \cite{williamsDA}. 

We refer to \cite{BBY}, or \cite[Section 1.2.4]{BM_book}, for the details on the DA construction. Here we just apply it to build our example:

\begin{construction}\label{construction: gluing suspension}
     Let $\varphi \colon M \to M$ be a suspension Anosov flow. Let $\alpha_1, \alpha_2$ be two distinct periodic orbits of $\varphi$. Do an attracting DA bifurcation on $\alpha_1$ and a repelling one on $\alpha_2$.
Let $V_1$, resp.~$V_2$, be a tubular neighborhood of $\alpha_1$, resp.~$\alpha_2$, with boundary transverse to the DA flow. By \cite{BBY}, we can use an orientation-preserving diffeomorphism $f \colon \partial V_1 \to \partial V_2$ to glue along the boundary tori, obtaining a new (compact) manifold $M_f = M \slash f$, and an Anosov flow $\varphi_f$ on $M_f$. Since suspension flows are transitive, by \cite[Lemma 7.6]{BBY}, $\varphi_f$ is also transitive. 
\end{construction}

We show

\begin{theorem}\label{theorem: non-compact suspension type}
    For the flow $\varphi_f \colon M_f \to M_f$ built in Construction \ref{construction: gluing suspension}, there exists a non-compact cover $\overbar{M_f}$ of $M_f$ such that the lifted flow $\overbar{\varphi_f} \colon \overbar{M_f} \to \overbar{M_f}$ is transitive and every periodic orbit of $\overbar{\varphi_f}$ is alone in its free homotopy class.
\end{theorem}

The construction above can be easily extended to build many different examples satisfying the conclusion of the theorem. Note that, given the main result of Fenley in \cite{fenley2022nonrcoveredanosovflows}, it seems plausible that many (maybe all?) non-$\R$-covered Anosov flows on hyperbolic $3$-manifolds have the property that only finitely many periodic orbits are not alone in their free homotopy class.

\begin{proof}[Proof of Theorem \ref{theorem: non-compact suspension type}]
    By construction, $M_f$ has a transverse torus $\T = \partial V_1 \sqcup_f \partial V_2$. The stable foliation and the unstable foliation, when restricted to $T$, have two Reeb components with two closed leaves; see Figure \ref{fig: transverse torus} (a).
    On the orbit space level, projections of lifts of closed leaves from the same foliation are nonseparated in their leaf spaces, and projections of lifts of adjacent close leaves from different foliations make perfect fits. The closed leaves on the torus $\T$ come from the intersection of $\T$ with cylindrical stable and unstable leaves that we denote by $l_1^s, l_2^s$ and $l_1^u, l_2^u$ respectively. Lifting to the universal cover, there exists $g \in \pi_1(\T) < \pi_1(M_f)$ that fixes these leaves $\widetilde{l}^s_1, \widetilde{l}^s_2, \widetilde{l}^u_1, \widetilde{l}^u_2$, so $g$ fixes points $\widetilde{\alpha}_1, \widetilde{\alpha}_2, \widetilde{\beta}_1, \widetilde{\beta}_2$ on each of these leaves in $\cP_{\varphi_f}$. By Proposition \ref{prop: free homotopy}, these fixed points are corners of a $g$-invariant chain of lozenges, so they correspond to periodic orbits $\alpha_1, \alpha_2, \beta_1, \beta_2$ of $\varphi_f$ that are freely homotopic to each other (and to $g$ as unoriented curves). \par

    We claim that  
    \begin{claim}\label{claim: no other free homotopy}
        $\alpha_1, \alpha_2, \beta_1, \beta_2$ are the only periodic orbits of $\varphi_f$ that are not unique in their free homotopy classes. 
    \end{claim}
    \begin{proof}[Proof of Claim \ref{claim: no other free homotopy}]
        By Proposition \ref{prop: free homotopy}, there exists a chain of lozenge $\mathcal{C}$ containing $\alpha_1, \alpha_2, \beta_1, \beta_2$ as corners. Now, suppose that there exists a periodic orbit $\gamma$ of $\varphi_f$ such that its free homotopy class $[\gamma] \neq [g]$ and $\gamma$ is not unique in $[\gamma]$. We have the following cases.

\begin{figure}[h]
    \subfigure[]{\includegraphics[width = 0.37\textwidth]{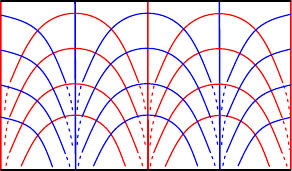}}
    \hspace{2cm}
    \subfigure[]{\includegraphics[width = 0.3\textwidth]{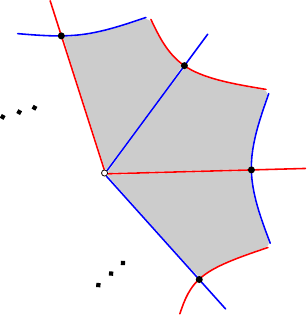}}
    \caption{Induced foliations on the transverse torus (a); and the projections in the orbit space (b). } 
    \label{fig: transverse torus}
\end{figure}

    \setcounter{case}{0}

    \begin{case}
        Suppose that $\gamma$ intersects $\T$ transversely. Then by Proposition \ref{prop: free homotopy}, the projection of $\widetilde{\gamma}$ (the lift of $\gamma$ to $\widetilde{M_f}$) to $\cP_{\varphi_f}$ is a corner of some lozenge. Since $\gamma$ crosses $\T$, this corner has to lie in the interior of a lozenge in $\mathcal{C}$. But this is impossible by the non-corner criterion (see \cite[Lemma 2.29]{BFM2022orbit} or \cite[Lemma 2.4.7]{BM_book}), which implies that any point in the interior of a lozenge of $\mathcal C$ cannot be a corner; see shaded region in Figure \ref{fig: transverse torus} (b).
    \end{case}

    \begin{case}
        Suppose that $\gamma$ does not cross $\T$, i.e., $\gamma \subset M_f\setminus \T$. 
        Let $\gamma'$ be another (possibly the same) orbit in the free homotopy class of $\gamma$. By the first part $\gamma'$ is also in $M_f\setminus \T$. 
        Consider a free homotopy between $\gamma$ and $\gamma'$. Either there exists one that stays in $M_f\setminus \T$ or all of them must cross the torus $\T$. In the first case, we would then get a nontrivial free homotopy between orbits of the DA of the suspension flow on $M$, which is absurd. In the second case, we would have that $\gamma$ is freely homotopic to an element in $\pi_1(\T)$, but by construction, the only periodic orbits of $\flow_f$ that are freely homotopic into $\T$ are $\alpha_1,\alpha_2,\beta_1$ and $\beta_2$, which proves the claim.\qedhere 
    \end{case}
\end{proof}

    By construction, the JSJ decomposition of $M_f$ is one atoroidal piece with two boundary tori. By Theorem \ref{thm: virtually special}, $\pi_1(M_f)$ is virtually RFRS, then by Theorem \ref{theorem: rfrs cover}, we know that there exists an infinite cover $\overbar{M_f}$ of $M_f$ on which the lifted flow $\overbar{\varphi_f}$ is transitive and that the action $\pi_1(\overbar{M_f}) \acts \PP_{\varphi_f}$ does not preserve any branching leaves. By Claim \ref{claim: no other free homotopy}, the only orbits that were not alone in their free homotopy classs in $M_f$ were on branching leaves, thus every periodic orbit of $\overbar{\flow}_f$ is alone in its free homotopy class.
\end{proof} 

In order to decide whether an example of the type built in Theorem \ref{theorem: non-compact suspension type} can admit a surface of section, a first step could be to check whether or not Fried's \emph{homological} characterization of suspensions (see \cite{Fried82}) is satisfied. In our context, it amounts to answering the following question:
\begin{question*}
    Let $\overbar{\flow}_f\colon \overbar M_f \to \overbar M_f$ be as in Theorem \ref{theorem: non-compact suspension type}. Does $\overbar{\flow}_f$ admit any homologically trivial periodic orbits?
\end{question*}

\bibliographystyle{alpha}
\bibliography{reference}

\end{document}